\documentclass[11pt, reqno]{amsart}
\UseRawInputEncoding
\usepackage{amsmath, amsthm, amscd, amsfonts, amssymb, graphicx, color}
\usepackage[bookmarksnumbered, colorlinks, plainpages]{hyperref}
\hypersetup{colorlinks=true,linkcolor=red, anchorcolor=green,
citecolor=cyan, urlcolor=red, filecolor=magenta, pdftoolbar=true}
\input{mathrsfs.sty}
\newtheorem{theorem}{Theorem}[section]

\newtheorem{corollary}[theorem]{Corollary}
\theoremstyle{definition}

\theoremstyle{remark}
\newtheorem{remark}[theorem]{Remark}
\numberwithin{equation}{section}

\begin{document}

\setcounter{page}{1}

\title[New estimates for numerical radius in $C^*$-algebras]
{New estimates for numerical radius in $C^*$-algebras}

\author[A.~Zamani]
{Ali Zamani}

\address{Department of Mathematics, Farhangian University, Tehran, Iran
\&
School of Mathematics and Computer Sciences, Damghan University, P.O.BOX 36715-364, Damghan, Iran}
\email{zamani.ali85@yahoo.com}

\begin{abstract}
Several numerical radius inequalities in the framework of $C^*$-algebras are proved in this paper.
These results, which are based on an extension of Buzano inequality for elements in a pre-Hilbert $C^*$-module,
generalize earlier numerical radius inequalities.
\end{abstract} \maketitle
\section{Introduction}
Let $\mathscr{A}$ be a $C^*$-algebra with unit $e$
and let $\mathcal{S}(\mathscr{A})$ be the normalized state space of $\mathscr{A}$,
i.e. $\mathcal{S}(\mathscr{A}) = \{\varphi \in \mathscr{A}':\, \varphi \geq 0 \,\,\,\mbox{and}\,\,\,\varphi(e) = 1\}$.
Recall that, for $a\in\mathscr{A}$, the numerical radius is defined by
$v(a) = \sup\left\{|\varphi(a)|: \, \varphi \in \mathcal{S}(\mathscr{A})\right\}$.
It is well-known that $v(\cdot)$ define a norm on $\mathscr{A}$, which is equivalent
to the $C^*$-norm $\|\!\cdot\!\|$.
In fact, for every $a\in \mathscr{A}$,
\begin{align}\label{numerical radius inequality}
\frac{1}{2}\|a\| \leq v(a)\leq \|a\|.
\end{align}
The inequalities in \eqref{numerical radius inequality} are sharp.
The first inequality becomes an equality if $a$ is $2$-nilpotent.
The second inequality becomes an equality if $a$ is normal.
A fundamental inequality for the numerical radius is the power inequality,
which says that for $a\in \mathscr{A}$, $v(a^n) \leq v^n(a)$ for all $n=1, 2, 3, \cdots$.
For a comprehensive account of the numerical radius in $C^*$-algebras, the reader is
referred to \cite{amz, Bo.Du, Bo.Ma, M.Z, Z.Positivity, Z.MIA}.

Pre-Hilbert $C^*$-modules were introduced by Paschke in \cite{Pas} and became a useful tool in various
aspects of noncommutative geometry and topology.
The idea behind this notion is to generalize the notion of a pre-Hilbert space by replacing
the field of the complex numbers, both as scalars and as the range for the inner product,
by a $C^*$-algebra $\mathscr{A}$. For details about pre-Hilbert $C^*$--modules, we refer the reader to \cite{M.T}.

Let $\mathscr{X}$ be a pre-Hilbert $\mathscr{A}$-module with an $\mathscr{A}$-inner product $\langle \cdot, \cdot\rangle$
and let $\varphi$ be a positive linear functional on $\mathscr{A}$.
For $x, y, z \in\mathscr{X}$ with $\varphi(|z|^2) = 1$, very recently in \cite{Z.AFA} it has been shown that
\begingroup\makeatletter\def\f@size{10}\check@mathfonts
\begin{align}\label{G.B}
\Big|\varphi(\langle x, z\rangle)\varphi(\langle y, z\rangle)\Big| \leq \frac{1}{|\alpha|}\Big(\max\{1, |\alpha -1|\}
\sqrt{\varphi(|x|^2)}\sqrt{\varphi(|y|^2)} + \Big|\varphi(\langle x, y\rangle)\Big|\Big),
\end{align}
\endgroup
for any $\alpha \in \mathbb{C}\setminus\{0\}$. Here, the symbol $|x|^2$ stands for $\langle x, x\rangle$.
In particular (taking $\alpha =2$),
\begingroup\makeatletter\def\f@size{10}\check@mathfonts
\begin{align}\label{C.B}
\Big|\varphi(\langle x, z\rangle)\varphi(\langle y, z\rangle)\Big| \leq \frac{1}{2}\Big(
\sqrt{\varphi(|x|^2)}\sqrt{\varphi(|y|^2)} + \Big|\varphi(\langle x, y\rangle)\Big|\Big).
\end{align}
\endgroup
This inequality can be considered as a generalization of the classical Buzano inequality (see \cite{Buz}) in pre-Hilbert $C^*$-modules.
Also, by letting $z=\varphi^{-\frac{1}{2}}(|y|^2)y$ in \eqref{C.B},
we get a useful version of the Cauchy--Schwarz inequality as follows:
\begingroup\makeatletter\def\f@size{10}\check@mathfonts
\begin{align}\label{C.S}
\big|\varphi(\langle x, y\rangle)\big|\leq \sqrt{\varphi(|x|^2)}\sqrt{\varphi(|y|^2)}.
\end{align}
\endgroup
This Cauchy--Schwarz inequality has a lot of elegant applications in the geometry of pre-Hilbert $C^*$-modules,
see, for instance, \cite{A.R.LAMA, W.Z.LAMA, Z.M-IM}. Some other extensions of the Cauchy–-Schwarz inequality in the framework of $C^*$-algebras and pre-Hilbert $C^*$-modules can also be found in \cite{A.B.F.M.AFA, A.B.M, F.F.M.P.S, F.F.M.S, F.F.S.1, F.F.S.2, GG.Dr, I.V, K.D.M.Filomat, Z.AFA, Z.LAMA}
and references therein.

It is our aim in this paper to present new upper bounds for the numerical radius in $C^*$-algebras.
We firstly propose an extension of the Buzano inequality for elements in a pre-Hilbert $C^*$-module.
We then use this inequality to obtain some new estimates for numerical radius in $C^*$-algebras.
Our results extend or improve some theorems in the literature.
\section{Main results}
We begin with an extension of the Buzano inequality in pre-Hilbert $C^*$-modules.
\begin{theorem}\label{T.2.1}
Let $\mathbb{D}$ be a subset of $\mathbb{R}$ and let $f: \mathbb{D}\rightarrow [0, \infty)$
be a mapping such that $f(t)+f(1-t)=1$ for all $t\in\mathbb{D}$.
Let $\mathscr{X}$ be a pre-Hilbert $\mathscr{A}$-module and let $\varphi$ be a positive linear functional on $\mathscr{A}$.
If $x_1, x_2, \ldots, x_n, z \in\mathscr{X}$ such that $\varphi(|z|^2) = 1$,
then for any $\xi\in\mathbb{D}$ and $\alpha \in \mathbb{C}\setminus\{0\}$
\begingroup\makeatletter\def\f@size{10}\check@mathfonts
\begin{align*}
\left|\prod_{i=1}^{n}\varphi(\langle x_i, z\rangle)\right|^2
&\leq \frac{\max\{1, |\alpha -1|^2\}}{|\alpha|^2}\prod_{i=1}^{n}\varphi(|x_i|^2)
+ \frac{f(1-\xi)}{|\alpha|^2}\left|\varphi(\langle x_1, x_2\rangle)\prod_{i=3}^{n}\varphi(\langle x_i, z\rangle)\right|^2
\\& \qquad + \frac{f(\xi) + 2\max\{1, |\alpha -1|\}}{|\alpha|^2}\prod_{i=1}^{n}\sqrt{\varphi(|x_i|^2)}
\left|\varphi(\langle x_1, x_2\rangle)\prod_{i=3}^{n}\varphi(\langle x_i, z\rangle)\right|.
\end{align*}
\endgroup
\end{theorem}
\begin{proof}
Let $\xi\in\mathbb{D}$ and $\alpha \in \mathbb{C}\setminus\{0\}$. Put $y = \left(\displaystyle{\prod_{i=3}^{n}}\varphi(\langle x_i, z\rangle)\right)x_2$.
By \eqref{G.B} and the Cauchy--Schwarz inequality \eqref{C.S} we have
\begingroup\makeatletter\def\f@size{9}\check@mathfonts
\begin{align*}
\left|\prod_{i=1}^{n}\varphi(\langle x_i, z\rangle)\right|
&= \Big|\varphi(\langle x_1, z\rangle)\varphi(\langle y, z\rangle)\Big|
\\& \leq \frac{1}{|\alpha|}\left(\max\{1, |\alpha -1|\}
\sqrt{\varphi(|x_1|^2)}\sqrt{\varphi(|y|^2)} + \left|\varphi(\langle x_1, y\rangle)\right|\right)
\\& = \frac{1}{|\alpha|}\left(\max\{1, |\alpha -1|\}\sqrt{\varphi(|x_1|^2)}\sqrt{\varphi(|x_2|^2)}
\prod_{i=3}^{n}\left|\varphi(\langle x_i, z\rangle)\right| + \left|\varphi(\langle x_1, x_2\rangle)\prod_{i=3}^{n}\varphi(\langle x_i, z\rangle)\right|\right)
\\& \leq \frac{1}{|\alpha|}\left(\max\{1, |\alpha -1|\}\sqrt{\varphi(|x_1|^2)}\sqrt{\varphi(|x_2|^2)}
\prod_{i=3}^{n}\sqrt{\varphi(|x_i|^2)}\sqrt{\varphi(|z|^2)} + \left|\varphi(\langle x_1, x_2\rangle)\prod_{i=3}^{n}\varphi(\langle x_i, z\rangle)\right|\right)
\\&=\frac{1}{|\alpha|}\left(\max\{1, |\alpha -1|\}
\prod_{i=1}^{n}\sqrt{\varphi(|x_i|^2)} + \left|\varphi(\langle x_1, x_2\rangle)\prod_{i=3}^{n}\varphi(\langle x_i, z\rangle)\right|\right),
\end{align*}
\endgroup
and hence,
\begingroup\makeatletter\def\f@size{9}\check@mathfonts
\begin{align*}
\left|\prod_{i=1}^{n}\varphi(\langle x_i, z\rangle)\right|^2 &\leq \frac{1}{|\alpha|^2}\left(\max\{1, |\alpha -1|\}
\prod_{i=1}^{n}\sqrt{\varphi(|x_i|^2)} + \left|\varphi(\langle x_1, x_2\rangle)\prod_{i=3}^{n}\varphi(\langle x_i, z\rangle)\right|\right)^2
\\& = \frac{\max\{1, |\alpha -1|^2\}}{|\alpha|^2}\prod_{i=1}^{n}\varphi(|x_i|^2)
+ \frac{1}{|\alpha|^2}\left|\varphi(\langle x_1, x_2\rangle)\prod_{i=3}^{n}\varphi(\langle x_i, z\rangle)\right|^2
\\& \qquad + \frac{2\max\{1, |\alpha -1|\}}{|\alpha|^2}\prod_{i=1}^{n}\sqrt{\varphi(|x_i|^2)}
\left|\varphi(\langle x_1, x_2\rangle)\prod_{i=3}^{n}\varphi(\langle x_i, z\rangle)\right|
\\& = \frac{\max\{1, |\alpha -1|^2\}}{|\alpha|^2}\prod_{i=1}^{n}\varphi(|x_i|^2)
\\& \qquad + \frac{f(1-\xi)}{|\alpha|^2}\left|\varphi(\langle x_1, x_2\rangle)\prod_{i=3}^{n}\varphi(\langle x_i, z\rangle)\right|^2 + \frac{f(\xi)}{|\alpha|^2}\left|\varphi(\langle x_1, x_2\rangle)\prod_{i=3}^{n}\varphi(\langle x_i, z\rangle)\right|^2
\\& \qquad \quad + \frac{2\max\{1, |\alpha -1|\}}{|\alpha|^2}\prod_{i=1}^{n}\sqrt{\varphi(|x_i|^2)}
\left|\varphi(\langle x_1, x_2\rangle)\prod_{i=3}^{n}\varphi(\langle x_i, z\rangle)\right|
\\&= \frac{\max\{1, |\alpha -1|^2\}}{|\alpha|^2}\prod_{i=1}^{n}\varphi(|x_i|^2)+ \frac{f(1-\xi)}{|\alpha|^2}\left|\varphi(\langle x_1, x_2\rangle)\prod_{i=3}^{n}\varphi(\langle x_i, z\rangle)\right|^2
\\& \qquad + \frac{f(\xi)}{|\alpha|^2}\big|\varphi(\langle x_1, x_2\rangle)\big|\prod_{i=3}^{n}\big|\varphi(\langle x_i, z\rangle)\big|
\left|\varphi(\langle x_1, x_2\rangle)\prod_{i=3}^{n}\varphi(\langle x_i, z\rangle)\right|
\\& \qquad \quad + \frac{2\max\{1, |\alpha -1|\}}{|\alpha|^2}\prod_{i=1}^{n}\sqrt{\varphi(|x_i|^2)}
\left|\varphi(\langle x_1, x_2\rangle)\prod_{i=3}^{n}\varphi(\langle x_i, z\rangle)\right|
\\&\leq \frac{\max\{1, |\alpha -1|^2\}}{|\alpha|^2}\prod_{i=1}^{n}\varphi(|x_i|^2)+ \frac{f(1-\xi)}{|\alpha|^2}\left|\varphi(\langle x_1, x_2\rangle)\prod_{i=3}^{n}\varphi(\langle x_i, z\rangle)\right|^2
\\& \qquad + \frac{f(\xi)}{|\alpha|^2}\sqrt{\varphi(|x_1|^2)}\sqrt{\varphi(|x_2|^2)}\prod_{i=3}^{n}\sqrt{\varphi(|x_i|^2)}\sqrt{\varphi(|z|^2)}
\left|\varphi(\langle x_1, x_2\rangle)\prod_{i=3}^{n}\varphi(\langle x_i, z\rangle)\right|
\\& \qquad \quad + \frac{2\max\{1, |\alpha -1|\}}{|\alpha|^2}\prod_{i=1}^{n}\sqrt{\varphi(|x_i|^2)}
\left|\varphi(\langle x_1, x_2\rangle)\prod_{i=3}^{n}\varphi(\langle x_i, z\rangle)\right|
\\&= \frac{\max\{1, |\alpha -1|^2\}}{|\alpha|^2}\prod_{i=1}^{n}\varphi(|x_i|^2)
+ \frac{f(1-\xi)}{|\alpha|^2}\left|\varphi(\langle x_1, x_2\rangle)\prod_{i=3}^{n}\varphi(\langle x_i, z\rangle)\right|^2
\\& \qquad + \frac{f(\xi) + 2\max\{1, |\alpha -1|\}}{|\alpha|^2}\prod_{i=1}^{n}\sqrt{\varphi(|x_i|^2)}
\left|\varphi(\langle x_1, x_2\rangle)\prod_{i=3}^{n}\varphi(\langle x_i, z\rangle)\right|.
\end{align*}
\endgroup
\end{proof}
As an immediate consequence of Theorem \ref{T.2.1}, we have the following result.
\begin{corollary}\label{C.2.2}
Let $\mathscr{X}$ be a pre-Hilbert $\mathscr{A}$-module and let $\varphi$ be a positive linear functional on $\mathscr{A}$.
If $x_1, x_2, \ldots, x_n, z \in\mathscr{X}$ such that $\varphi(|z|^2) = 1$, then for any $\alpha \in \mathbb{C}\setminus\{0\}$
\begingroup\makeatletter\def\f@size{10}\check@mathfonts
\begin{align}\label{T.2.1.I.1}
\Big|\prod_{i=1}^{n}\varphi(\langle x_i, z\rangle)\Big| \leq \frac{1}{|\alpha|}\Big(\max\{1, |\alpha -1|\}
\prod_{i=1}^{n}\sqrt{\varphi(|x_i|^2)} + \Big|\varphi(\langle x_1, x_2\rangle)\prod_{i=3}^{n}\varphi(\langle x_i, z\rangle)\Big|\Big).
\end{align}
\endgroup
In particular,
\begingroup\makeatletter\def\f@size{10}\check@mathfonts
\begin{align}\label{T.2.1.I.2}
\Big|\prod_{i=1}^{n}\varphi(\langle x_i, z\rangle)\Big| \leq \frac{1}{2}\Big(
\prod_{i=1}^{n}\sqrt{\varphi(|x_i|^2)} + \Big|\varphi(\langle x_1, x_2\rangle)\prod_{i=3}^{n}\varphi(\langle x_i, z\rangle)\Big|\Big).
\end{align}
\endgroup
\end{corollary}
\begin{remark}\label{R.2.3}
Since the only state of the $C^*$-algebra $\mathbb{C}$ of all complex numbers is the
identity map, in the case of a pre-Hilbert space $\big(\mathscr{H}, [\cdot, \cdot]\big)$ (regarded as a pre-Hilbert $\mathbb{C}$-module)
and for any $x_1, x_2, \ldots, x_n, z \in \mathscr{H}$ with $\|z\|=1$, the inequality \eqref{T.2.1.I.2} in Corollary \ref{C.2.2} reads as
\begingroup\makeatletter\def\f@size{10}\check@mathfonts
\begin{align*}
\Big|\prod_{i=1}^{n}[x_i, z]\Big| \leq \frac{1}{2}\Big(
\prod_{i=1}^{n}\|x_i\| + \Big|[x_1, x_2]\prod_{i=3}^{n}[x_i, z]\Big|\Big),
\end{align*}
\endgroup
which is an extension of the classical Buzano inequality in pre-Hilbert space $\mathscr{H}$.
\end{remark}
Two other consequences of Theorem \ref{T.2.1} can be stated as follows.
\begin{corollary}\label{C.2.4}
Let $\mathscr{X}$ be a pre-Hilbert $\mathscr{A}$-module and let $\varphi$ be a positive linear functional on $\mathscr{A}$.
If $x_1, x_2, \ldots, x_n, z \in\mathscr{X}$ such that $\varphi(|z|^2) = 1$,
then for any $\zeta\geq0$ and $\alpha \in \mathbb{C}\setminus\{0\}$
\begingroup\makeatletter\def\f@size{10}\check@mathfonts
\begin{align*}
\Big|\prod_{i=1}^{n}\varphi(\langle x_i, z\rangle)\Big|^2
&\leq \frac{\max\{1, |\alpha -1|^2\}}{|\alpha|^2}\prod_{i=1}^{n}\varphi(|x_i|^2)
+ \frac{1}{|\alpha|^2(1+\zeta)}\Big|\varphi(\langle x_1, x_2\rangle)\prod_{i=3}^{n}\varphi(\langle x_i, z\rangle)\Big|^2\nonumber
\\& \qquad + \frac{\zeta +2(1+\zeta)\max\{1, |\alpha -1|\}}{|\alpha|^2(1+\zeta)}\prod_{i=1}^{n}\sqrt{\varphi(|x_i|^2)}
\Big|\varphi(\langle x_1, x_2\rangle)\prod_{i=3}^{n}\varphi(\langle x_i, z\rangle)\Big|.
\end{align*}
\endgroup
\end{corollary}
\begin{proof}
Set $f(t)=t$ and $\xi=\frac{\zeta}{1+\zeta}$ in Theorem \ref{T.2.1}.
\end{proof}
\begin{corollary}\label{C.2.5}
Let $\mathscr{X}$ be a pre-Hilbert $\mathscr{A}$-module and let $\varphi$ be a positive linear functional on $\mathscr{A}$.
If $x_1, x_2, \ldots, x_n, z \in\mathscr{X}$ such that $\varphi(|z|^2) = 1$,
then for any $\eta\in[-\frac{1}{2}, \frac{3}{2}]$ and $\alpha \in \mathbb{C}\setminus\{0\}$
\begingroup\makeatletter\def\f@size{10}\check@mathfonts
\begin{align*}
\Big|\prod_{i=1}^{n}\varphi(\langle x_i, z\rangle)\Big|^2
&\leq \frac{\max\{1, |\alpha -1|^2\}}{|\alpha|^2}\prod_{i=1}^{n}\varphi(|x_i|^2)
+ \frac{3-2\eta}{4|\alpha|^2}\Big|\varphi(\langle x_1, x_2\rangle)\prod_{i=3}^{n}\varphi(\langle x_i, z\rangle)\Big|^2\nonumber
\\& \qquad + \frac{1+2\eta +8\max\{1, |\alpha -1|\}}{4|\alpha|^2}\prod_{i=1}^{n}\sqrt{\varphi(|x_i|^2)}
\Big|\varphi(\langle x_1, x_2\rangle)\prod_{i=3}^{n}\varphi(\langle x_i, z\rangle)\Big|.
\end{align*}
\endgroup
\end{corollary}
\begin{proof}
Set $f(t)=\frac{1+2t}{4}$ and $\xi=\eta$ in Theorem \ref{T.2.1}.
\end{proof}
In the following, we utilize the above extensions of the Buzano inequality in pre-Hilbert $C^*$-modules to
present new upper bounds for the numerical radius in $C^*$-algebras.
First, we apply Corollary \ref{C.2.2} to prove the following result.
\begin{theorem}\label{T.3.4}
Let $\mathscr{A}$ be a $C^*$-algebra with unit $e$ and let $a\in \mathscr{A}$. Then
\begin{align*}
v^n(a)\leq \sum_{i=1}^{n-1}\frac{1}{2^i}\|a^i\|{\|a\|}^{n-i}+\frac{1}{2^{n-1}}v(a^n),
\end{align*}
for all $n=2, 3, \ldots$.
\end{theorem}
\begin{proof}
Let $n\in\mathbb{N}$ and $n\geq2$. We will show that for any $\varphi \in \mathcal{S}(\mathscr{A})$,
\begin{align}\label{T.3.4.I.1}
\left|\varphi(a)\right|^n\leq \sum_{i=1}^{n-1}\frac{1}{2^i}\sqrt{\varphi(|a^i|^2)}\sqrt{\varphi^{n-i}(|a^*|^2)}+\frac{1}{2^{n-1}}\left|\varphi(a^n)\right|.
\end{align}
By Corollary \ref{C.2.2} with $\mathscr{X} = \mathscr{A}$ we have
\begingroup\makeatletter\def\f@size{10}\check@mathfonts
\begin{align*}
\left|\varphi(a)\right|^n&=\left|\varphi(\langle a, e\rangle)\varphi(\langle a^*, e\rangle)\varphi^{n-2}(\langle a^*, e\rangle)\right|
\\&\leq \frac{1}{2}\sqrt{\varphi(|a|^2)}\sqrt{\varphi^{n-1}(|a^*|^2)}
+ \frac{1}{2}\left|\varphi(\langle a, a^*\rangle)\varphi^{n-2}(\langle a^*, e\rangle)\right|
\\& = \frac{1}{2}\sqrt{\varphi(|a|^2)}\sqrt{\varphi^{n-1}(|a^*|^2)}
+ \frac{1}{2}\left|\varphi(\langle a^2, e\rangle)\varphi(\langle a^*, e\rangle)\varphi^{n-3}(\langle a^*, e\rangle)\right|
\\& \leq \frac{1}{2}\sqrt{\varphi(|a|^2)}\sqrt{\varphi^{n-1}(|a^*|^2)}
+ \frac{1}{4}\sqrt{\varphi(|a^2|^2)}\sqrt{\varphi^{n-2}(|a^*|^2)}
\\& \qquad + \frac{1}{4}\left|\varphi(\langle a^2, a^*\rangle)\varphi^{n-3}(\langle a^*, e\rangle)\right|
\\& = \frac{1}{2}\sqrt{\varphi(|a|^2)}\sqrt{\varphi^{n-1}(|a^*|^2)}
+ \frac{1}{4}\sqrt{\varphi(|a^2|^2)}\sqrt{\varphi^{n-2}(|a^*|^2)}
\\& \qquad+ \frac{1}{4}\left|\varphi(\langle a^3, e\rangle)\varphi(\langle a^*, e\rangle)\varphi^{n-4}(\langle a^*, e\rangle)\right|
\\& \leq \frac{1}{2}\sqrt{\varphi(|a|^2)}\sqrt{\varphi^{n-1}(|a^*|^2)}
+ \frac{1}{4}\sqrt{\varphi(|a^2|^2)}\sqrt{\varphi^{n-2}(|a^*|^2)}
\\& \qquad + \frac{1}{8}\sqrt{\varphi(|a^3|^2)}\sqrt{\varphi^{n-3}(|a^*|^2)}
+\frac{1}{8}\left|\varphi(\langle a^3, a^*\rangle)\varphi^{n-4}(\langle a^*, e\rangle)\right|.
\end{align*}
\endgroup
If we continue in this way, we obtain
\begingroup\makeatletter\def\f@size{10}\check@mathfonts
\begin{align*}
\left|\varphi(a)\right|^n&\leq \frac{1}{2}\sqrt{\varphi(|a|^2)}\sqrt{\varphi^{n-1}(|a^*|^2)}
+ \frac{1}{4}\sqrt{\varphi(|a^2|^2)}\sqrt{\varphi^{n-2}(|a^*|^2)}
\\& \qquad + \frac{1}{8}\sqrt{\varphi(|a^3|^2)}\sqrt{\varphi^{n-3}(|a^*|^2)} + \frac{1}{16}\sqrt{\varphi(|a^4|^2)}\sqrt{\varphi^{n-4}(|a^*|^2)}
\\& \qquad \quad +\cdots + \frac{1}{2^{n-1}}\sqrt{\varphi(|a^{n-1}|^2)}\sqrt{\varphi(|a^*|^2)} +
\frac{1}{2^{n-1}}\left|\varphi(\langle a^{n-1}, a^*\rangle)\right|,
\end{align*}
\endgroup
which gives \eqref{T.3.4.I.1}.
Now, by taking the supremum over all $\varphi \in \mathcal{S}(\mathscr{A})$ in \eqref{T.3.4.I.1}, the desired inequality follows.
\end{proof}
\begin{corollary}\label{T.3.3.C.2}
Let $\mathscr{A}$ be a unital $C^*$-algebra and let $a\in \mathscr{A}$. Then
\begin{align*}
v^n(a)\leq \frac{1}{2^{n-1}-1}\sum_{i=1}^{n-1}2^{n-i-1}{\|a\|}^{n-i}\|a^i\|.
\end{align*}
for all $n=2, 3, \ldots$.
\end{corollary}
\begin{proof}
The proof follows Theorem \ref{T.3.4} and the power inequality for the numerical radius.
\end{proof}
\begin{remark}\label{T.3.4.R.1}
For $n=2, 3, \ldots$ we have
\begin{align*}
&\sum_{i=1}^{n-1}\frac{1}{2^i}{\|a\|}^{n-i}\|a^i\|+\frac{1}{2^{n-1}}v(a^n)
\\& \quad \leq\sum_{i=1}^{n-1}\frac{1}{2^i}{\|a\|}^{n}+\frac{1}{2^{n-1}}\|a^n\|
\\& \quad \leq \left(1-\frac{1}{2^{n-1}}\right){\|a\|}^n+\frac{1}{2^{n-1}}{\|a\|}^n= \|a\|^n.
\end{align*}
Therefore, the inequality in Theorem \ref{T.3.4} is an improvement of the second inequality in \eqref{numerical radius inequality}.
\end{remark}
Recall that the spectral radius of $a\in\mathscr{A}$ is the number
$r(a) = \sup\left\{|\lambda|: \, \lambda\in \sigma(a)\right\}$,
where $\sigma(a)$ is the spectrum of $a$.
It is well-known that the spectral radius $r(a)$ of $a$ satisfies
$r(a)= \displaystyle{\lim_{n\rightarrow +\infty}}{\|a^n\|}^{\frac{1}{n}}$.
\begin{corollary}\label{T.3.3.C.1}
Let $\mathscr{A}$ be a unital $C^*$-algebra and let $a\in \mathscr{A}$. If $v(a)=\|a\|$, then $v(a)=r(a)$.
\end{corollary}
\begin{proof}
Let $v(a)=\|a\|$. By the inequalities in Remark \ref{T.3.4.R.1} it follows that
\begin{align*}
v^n(a)=v(a^n)=\|a^n\|=\|a\|^n \qquad (n=2, 3, \ldots).
\end{align*}
This implies
$v(a)= \displaystyle{\lim_{n\rightarrow +\infty}}{\|a^n\|}^{\frac{1}{n}}=r(a)$.
\end{proof}
In the following theorem we state another upper bound for the numerical radius in $C^*$-algebras.
\begin{theorem}\label{T.3.1}
Let $\mathscr{A}$ be a $C^*$-algebra with unit $e$ and let $a\in \mathscr{A}$.
For any $\alpha, \beta \in \mathbb{C}\setminus\{0\}$,
\begingroup\makeatletter\def\f@size{10}\check@mathfonts
\begin{align*}
v^3(a) \leq \frac{\max\{1, |\alpha -1|\}}{2|\alpha|}
\left\|\,|a^*|^2+|a|^2\right\|\,\|a\|
+ \frac{\max\{1, |\beta -1|\}}{|\alpha\beta|}
\|a^2\|\,\|a\| + \frac{1}{|\alpha\beta|}v(a^3).
\end{align*}
\endgroup
\end{theorem}
\begin{proof}
Let $\alpha, \beta\in \mathbb{C}\setminus\{0\}$. For any $\varphi \in \mathcal{S}(\mathscr{A})$,
by Corollary \ref{C.2.2} with $\mathscr{X} = \mathscr{A}$ and the arithmetic–-geometric mean inequality we have
\begingroup\makeatletter\def\f@size{10}\check@mathfonts
\begin{align*}
\left|\varphi(a)\right|^3&=\left|\varphi(\langle a^*, e\rangle)\varphi(\langle a, e\rangle)\varphi(\langle a, e\rangle)\right|
\\&\leq \frac{1}{|\alpha|}\left(\max\{1, |\alpha -1|\}
\sqrt{\varphi(|a^*|^2)}\sqrt{\varphi(|a|^2)}\sqrt{\varphi(|a|^2)}
+ \left|\varphi(\langle a^*, a\rangle)\varphi(\langle a, e\rangle)\right|\right)
\\& = \frac{\max\{1, |\alpha -1|\}}{|\alpha|}
\sqrt{\varphi(|a^*|^2)}\sqrt{\varphi(|a|^2)}\sqrt{\varphi(|a|^2)}
+ \frac{1}{|\alpha|}\left|\varphi(\langle (a^*)^2, e\rangle)\varphi(\langle a, e\rangle)\right|
\\& \leq \frac{\max\{1, |\alpha -1|\}}{2|\alpha|}
\left(\varphi(|a^*|^2)+\varphi(|a|^2)\right)\sqrt{\varphi(|a|^2)}
\\& \qquad + \frac{1}{|\alpha||\beta|}\left(\max\{1, |\beta -1|\}
\sqrt{\varphi(|(a^*)^2|^2)}\sqrt{\varphi(|a|^2)}
+ \left|\varphi(\langle (a^*)^2, a\rangle)\right|\right)
\\& = \frac{\max\{1, |\alpha -1|\}}{2|\alpha|}
\varphi\left(|a^*|^2+|a|^2\right)\sqrt{\varphi(|a|^2)}
\\& \qquad + \frac{\max\{1, |\beta -1|\}}{|\alpha\beta|}
\sqrt{\varphi(|(a^*)^2|^2)}\sqrt{\varphi(|a|^2)}
+ \frac{1}{|\alpha\beta|}\left|\varphi(a^3)\right|
\\& \leq \frac{\max\{1, |\alpha -1|\}}{2|\alpha|}
\left\|\,|a^*|^2+|a|^2\right\|\,\|a\|
+ \frac{\max\{1, |\beta -1|\}}{|\alpha\beta|}
\|a^2\|\,\|a\| + \frac{1}{|\alpha\beta|}v(a^3),
\end{align*}
\endgroup
and so
\begingroup\makeatletter\def\f@size{10}\check@mathfonts
\begin{align*}
\left|\varphi(a)\right|^3\leq \frac{\max\{1, |\alpha -1|\}}{2|\alpha|}
\left\|\,|a^*|^2+|a|^2\right\|\,\|a\|
+ \frac{\max\{1, |\beta -1|\}}{|\alpha\beta|}
\|a^2\|\,\|a\| + \frac{1}{|\alpha\beta|}v(a^3).
\end{align*}
\endgroup
Now, by taking the supremum over all $\varphi \in \mathcal{S}(\mathscr{A})$ in the last inequality, the desired inequality follows.
\end{proof}
Here are some immediate consequences of Theorem \ref{T.3.1}.
\begin{corollary}\label{T.3.1.C.1}
Let $\mathscr{A}$ be a unital $C^*$-algebra and let $a\in \mathscr{A}$. Then
\begin{align*}
v^3(a) \leq \frac{1}{4}\left\|\,|a^*|^2 + |a|^2\right\|\,\|a\| + \frac{1}{4}\left\|a^2\right\|\,\|a\| + \frac{1}{4}v(a^3).
\end{align*}
\end{corollary}
\begin{proof}
The proof follows Theorem \ref{T.3.1} by letting $\alpha=\beta=2$.
\end{proof}
\begin{corollary}\label{T.3.1.C.3}
Let $\mathscr{A}$ be a unital $C^*$-algebra and let $a\in \mathscr{A}$. Then
\begin{align*}
v^3(a) \leq \frac{1}{2}\left\|\,|a^*|^2 + |a|^2\right\|\,\|a\|.
\end{align*}
\end{corollary}
\begin{proof}
The proof follows Theorem \ref{T.3.1} by letting $\alpha=\beta=n$ and $n\rightarrow\infty$.
\end{proof}
\begin{corollary}\label{T.3.1.C.4}
Let $\mathscr{A}$ be a unital $C^*$-algebra and let $a\in \mathscr{A}$. Then
\begin{align*}
v^3(a) \leq \frac{1}{4}\left\|\,|a^*|^2 + |a|^2\right\|\,\|a\| +\frac{1}{2}\left\|a^2\right\|\|a\|.
\end{align*}
\end{corollary}
\begin{proof}
The proof follows Theorem \ref{T.3.1} by letting $\alpha=2$, $\beta=n$ and $n\rightarrow\infty$.
\end{proof}
As an immediate consequence of Corollary \ref{T.3.1.C.1} and the power inequality for the numerical radius, we have the following result.
\begin{corollary}\label{T.3.1.C.2}
Let $\mathscr{A}$ be a unital $C^*$-algebra and let $a\in \mathscr{A}$. Then
\begin{align*}
v^3(a) \leq \frac{1}{3}\left\||a^*|^2 + |a|^2\right\|\|a\| + \frac{1}{3}\left\|a^2\right\|\|a\|.
\end{align*}
\end{corollary}
\begin{remark}\label{T.3.1.R.1}
The inequalities in Corollaries \ref{T.3.1.C.1}-\ref{T.3.1.C.2} are improvements of the second inequality in \eqref{numerical radius inequality}.
Indeed, for instance in Corollary \ref{T.3.1.C.1}, we have
\begin{align*}
\frac{1}{4}&\left\|\,|a^*|^2 + |a|^2\right\|\,\|a\| + \frac{1}{4}\left\|a^2\right\|\,\|a\| + \frac{1}{4}v(a^3)
\\& \leq \frac{1}{4}\left(\left\|\,|a^*|^2\right\| + \left\|\,|a|^2\right\|\right)\|a\| + \frac{1}{4}\|a\|^2\|a\| + \frac{1}{4}\left\|a^3\right\|
\\& \leq \frac{2}{4}\|a\|^3 + \frac{1}{4}\|a\|^3 + \frac{1}{4}\|a\|^3 = \|a\|^3.
\end{align*}
\end{remark}
\begin{theorem}\label{T.3.2}
Let $\mathscr{A}$ be a $C^*$-algebra with unit $e$ and let $a\in \mathscr{A}$.
For any $\alpha \in \mathbb{C}\setminus\{0\}$,
\begingroup\makeatletter\def\f@size{10}\check@mathfonts
\begin{align*}
v^3(a) \leq \frac{1}{|\alpha|}\left(\max\{1, |\alpha -1|\}\|a\|^3+\min\{v(a^*a^2), v(a^2a^*)\}\right).
\end{align*}
\endgroup
\end{theorem}
\begin{proof}
Let $\alpha\in \mathbb{C}\setminus\{0\}$. For any $\varphi \in \mathcal{S}(\mathscr{A})$,
by \eqref{G.B} and the Cauchy--Schwarz inequality \eqref{C.S} we have
\begingroup\makeatletter\def\f@size{10}\check@mathfonts
\begin{align*}
\left|\varphi(a)\right|^3&=\left|\varphi(\langle a, e\rangle)\right|^2\left|\varphi(a^*)\right|
\\& \leq \varphi(a^*a)\left|\varphi(a^*)\right|
\\& =\left|\varphi(\langle a^*a, e\rangle)\varphi(\langle a, e\rangle)\right|
\\&\leq \frac{1}{|\alpha|}\left(\max\{1, |\alpha -1|\}
\sqrt{\varphi(|a^*a|^2)}\sqrt{\varphi(|a|^2)}
+ \left|\varphi(\langle a^*a, a\rangle)\right|\right)
\\& = \frac{1}{|\alpha|}\left(\max\{1, |\alpha -1|\}
\sqrt{\varphi(|a^*a|^2)}\sqrt{\varphi(|a|^2)}
+ \left|\varphi(a^*a^2)\right|\right),
\end{align*}
\endgroup
and so
\begingroup\makeatletter\def\f@size{10}\check@mathfonts
\begin{align*}
\left|\varphi(a)\right|^3\leq \frac{1}{|\alpha|}\left(\max\{1, |\alpha -1|\}\|a\|^3 + v\left(a^*a^2\right)\right).
\end{align*}
\endgroup
By taking the supremum over all $\varphi \in \mathcal{S}(\mathscr{A})$ in the last inequality, we obtain
\begingroup\makeatletter\def\f@size{10}\check@mathfonts
\begin{align}\label{T.3.2.I.1}
v^3(a)\leq \frac{1}{|\alpha|}\left(\max\{1, |\alpha -1|\}\|a\|^3 + v\left(a^*a^2\right)\right).
\end{align}
\endgroup
By a similar argument, we also have
\begingroup\makeatletter\def\f@size{10}\check@mathfonts
\begin{align}\label{T.3.2.I.2}
v^3(a)\leq \frac{1}{|\alpha|}\left(\max\{1, |\alpha -1|\}\|a\|^3 + v\left(a^2a^*\right)\right).
\end{align}
\endgroup
Utilizing \eqref{T.3.2.I.1} and \eqref{T.3.2.I.2}, we deduce the desired result.
\end{proof}
\begin{remark}\label{T.3.2.R.2}
Let $\mathscr{A}$ be a unital $C^*$-algebra and let $a\in \mathscr{A}$.
By letting $\alpha=2$ in Theorem \ref{T.3.2} we have
\begingroup\makeatletter\def\f@size{10}\check@mathfonts
\begin{align}\label{T.3.1.I.2}
v^3(a) \leq \frac{1}{2}\|a\|^3 + \frac{1}{2}\min\{v\left(a^*a^2\right), v\left(a^2a^*\right)\}.
\end{align}
\endgroup
Therefore if $v(a)=\|a\|$, then since $\min\{v\left(a^*a^2\right), v\left(a^2a^*\right)\}\leq {\|a\|}^3$,
by \eqref{T.3.1.I.2} it follows that
\begin{align*}
v\left(a^*a^2\right) = v\left(a^2a^*\right)=\|a\|^3.
\end{align*}
\end{remark}
In the following theorem we present a family of upper bounds for the numerical radius in $C^*$-algebras.
\begin{theorem}\label{T.3.3}
Let $\mathscr{A}$ be a $C^*$-algebra with unit $e$ and let $a\in \mathscr{A}$.
Let $j=1, 2$ and let $\mathbb{D}_j$ be a subset of $\mathbb{R}$. Let $f_j: \mathbb{D}_j\rightarrow [0, \infty)$
be a mapping such that $f_j(t)+f_j(1-t)=1$ for all $t\in\mathbb{D}_j$.
For any $\xi_j\in\mathbb{D}_j$ and $\alpha, \beta, \gamma \in \mathbb{C}\setminus\{0\}$
\begingroup\makeatletter\def\f@size{10}\check@mathfonts
\begin{align*}
v^6(a)&\leq \frac{\max\{1, |\alpha -1|^2\}}{2|\alpha|^2}\left\|\,|a^*|^4 + |a|^4\right\|\,{\|a\|}^2
+ \frac{f_1(1-\xi_1)\max\{1, |\beta -1|^2\}}{|\alpha\beta|^2}{\|a^2\|}^2{\|a\|}^2
\\& \qquad + \frac{f_1(1-\xi_1)f_2(1-\xi_2)}{|\alpha\beta|^2}v^2(a^3) + \frac{f_1(1-\xi_1)\left(f_2(\xi_2)+2\max\{1, |\beta -1|\}\right)}{|\alpha\beta|^2}\|a^2\|\,\|a\|\,v(a^3)
\\& \qquad \quad + \frac{\left(f_1(\xi_1)+2\max\{1, |\alpha -1|\}\right)\max\{1, |\gamma -1|\}}{2|\alpha|^2|\gamma|}\left\|\,|a^*|^2 + |a|^2\right\|\,\|a^2\|{\|a\|}^2
\\& \qquad \qquad + \frac{f_1(\xi_1)+2\max\{1, |\alpha -1|\}}{2|\alpha|^2|\gamma|}\left\|\,|a^*|^2 + |a|^2\right\|\,\|a\| \,v(a^3).
\end{align*}
\endgroup
\end{theorem}
\begin{proof}
Let $\xi_j\in\mathbb{D}_j$ and $\alpha, \beta, \gamma \in \mathbb{C}\setminus\{0\}$. For any $\varphi \in \mathcal{S}(\mathscr{A})$,
by Theorem \ref{T.2.1} with $\mathscr{X} = \mathscr{A}$, the arithmetic–-geometric mean inequality,
the Cauchy--Schwarz inequality \eqref{C.S} and \eqref{G.B} we have
\begingroup\makeatletter\def\f@size{10}\check@mathfonts
\begin{align*}
\left|\varphi(a)\right|^6&=\left|\varphi(\langle a^*, e\rangle)\varphi(\langle a, e\rangle)\varphi(\langle a, e\rangle)\right|^2
\\&\leq \frac{\max\{1, |\alpha -1|^2\}}{|\alpha|^2}\varphi(|a^*|^2)\varphi(|a|^2)\varphi(|a|^2)
+ \frac{f_1(1-\xi_1)}{|\alpha|^2}\left|\varphi(\langle a^*, a\rangle)\varphi(\langle a, e\rangle)\right|^2
\\& \qquad + \frac{f_1(\xi_1) + 2\max\{1, |\alpha -1|\}}{|\alpha|^2}\sqrt{\varphi(|a^*|^2)}\sqrt{\varphi(|a|^2)}\sqrt{\varphi(|a|^2)}
\left|\varphi(\langle a^*, a\rangle)\varphi(\langle a, e\rangle)\right|
\\& \leq \frac{\max\{1, |\alpha -1|^2\}}{2|\alpha|^2}\big(\varphi^2(|a^*|^2)+\varphi^2(|a|^2)\big)\varphi(|a|^2)
+ \frac{f_1(1-\xi_1)}{|\alpha|^2}\left|\varphi(\langle (a^*)^2, e\rangle)\varphi(\langle a, e\rangle)\right|^2
\\& \qquad + \frac{f_1(\xi_1) + 2\max\{1, |\alpha -1|\}}{2|\alpha|^2}\big(\varphi(|a^*|^2)+\varphi(|a|^2)\big)\sqrt{\varphi(|a|^2)}
\left|\varphi(\langle (a^*)^2, e\rangle)\varphi(\langle a, e\rangle)\right|
\\& \leq \frac{\max\{1, |\alpha -1|^2\}}{2|\alpha|^2}\big(\varphi(|a^*|^4)+\varphi(|a|^4)\big)\varphi(|a|^2)
\\& \qquad + \frac{f_1(1-\xi_1)}{|\alpha|^2}\Big(\frac{\max\{1, |\beta -1|^2\}}{|\beta|^2}\varphi(|(a^*)^2|^2)\varphi(|a|^2)
+ \frac{f_2(1-\xi_2)}{|\beta|^2}\left|\varphi(\langle (a^*)^2, a\rangle)\right|^2
\\& \qquad \quad + \frac{f_2(\xi_2) + 2\max\{1, |\beta -1|\}}{|\beta|^2}\sqrt{\varphi(|(a^*)^2|^2)}\sqrt{\varphi(|a|^2)}
\left|\varphi(\langle (a^*)^2, a\rangle)\right|\Big)
\\& \qquad + \frac{f_1(\xi_1) + 2\max\{1, |\alpha -1|\}}{2|\alpha|^2}\varphi\big(|a^*|^2+|a|^2\big)\sqrt{\varphi(|a|^2)}
\frac{1}{|\gamma|}\Big(\Big|\varphi(\langle (a^*)^2, a\rangle)\Big|
\\& \qquad \quad + \max\{1, |\gamma -1|\}\sqrt{\varphi(|(a^*)^2|^2)}\sqrt{\varphi(|a|^2)}\Big)
\\& = \frac{\max\{1, |\alpha -1|^2\}}{2|\alpha|^2}\varphi\big(|a^*|^4 +|a|^4\big)\varphi(|a|^2)
\\& \qquad + \frac{f_1(1-\xi_1)\max\{1, |\beta -1|^2\}}{|\alpha\beta|^2}\varphi(|(a^*)^2|^2)\varphi(|a|^2)
+ \frac{f_1(1-\xi_1)f_2(1-\xi_2)}{|\alpha\beta|^2}\left|\varphi(a^3)\right|^2
\\& \qquad \quad + \frac{f_1(1-\xi_1)\left(f_2(\xi_2)+2\max\{1, |\beta -1|\}\right)}{|\alpha\beta|^2}\sqrt{\varphi(|(a^*)^2|^2)}\sqrt{\varphi(|a|^2)}
\left|\varphi(a^3)\right|
\\& \qquad + \frac{\left(f_1(\xi_1)+2\max\{1, |\alpha -1|\}\right)\max\{1, |\gamma -1|\}}{2|\alpha|^2|\gamma|}\varphi\big(|a^*|^2+|a|^2\big)\sqrt{\varphi(|(a^*)^2|^2)}\varphi(|a|^2)
\\& \qquad \quad + \frac{f_1(\xi_1)+2\max\{1, |\alpha -1|\}}{2|\alpha|^2|\gamma|}\varphi\big(|a^*|^2+|a|^2\big)\sqrt{\varphi(|a|^2)}\left|\varphi(a^3)\right|.
\end{align*}
\endgroup
Therefore,
\begingroup\makeatletter\def\f@size{10}\check@mathfonts
\begin{align*}
\left|\varphi(a)\right|^6&\leq \frac{\max\{1, |\alpha -1|^2\}}{2|\alpha|^2}\left\|\,|a^*|^4 + |a|^4\right\|\,{\|a\|}^2
+ \frac{f_1(1-\xi_1)\max\{1, |\beta -1|^2\}}{|\alpha\beta|^2}{\|a^2\|}^2{\|a\|}^2
\\& \qquad + \frac{f_1(1-\xi_1)f_2(1-\xi_2)}{|\alpha\beta|^2}v^2(a^3) + \frac{f_1(1-\xi_1)\left(f_2(\xi_2)+2\max\{1, |\beta -1|\}\right)}{|\alpha\beta|^2}\|a^2\|\,\|a\|\,v(a^3)
\\& \qquad \quad + \frac{\left(f_1(\xi_1)+2\max\{1, |\alpha -1|\}\right)\max\{1, |\gamma -1|\}}{2|\alpha|^2|\gamma|}\left\|\,|a^*|^2 + |a|^2\right\|\,\|a^2\|{\|a\|}^2
\\& \qquad \qquad + \frac{f_1(\xi_1)+2\max\{1, |\alpha -1|\}}{2|\alpha|^2|\gamma|}\left\|\,|a^*|^2 + |a|^2\right\|\,\|a\| \,v(a^3).
\end{align*}
\endgroup
Now, by taking the supremum over all $\varphi \in \mathcal{S}(\mathscr{A})$ in the last inequality, the desired inequality follows.
\end{proof}
\begin{corollary}\label{T.3.3.C.1}
Let $\mathscr{A}$ be a unital $C^*$-algebra and let $a\in \mathscr{A}$. Then
\begingroup\makeatletter\def\f@size{10}\check@mathfonts
\begin{align*}
v^6(a)&\leq \frac{1}{8}\left\|\,|a^*|^4 + |a|^4\right\|\,{\|a\|}^2
+ \frac{1}{8}\left\|\,|a^*|^2 + |a|^2\right\|\,\|a^2\|{\|a\|}^2 + \frac{1}{16}{\|a^2\|}^2{\|a\|}^2
\\& \qquad + \frac{1}{8}\left\|\,|a^*|^2 + |a|^2\right\|\,\|a\| \,v(a^3) + \frac{1}{8}\|a^2\|\,\|a\|\,v(a^3) + \frac{1}{16}v^2(a^3).
\end{align*}
\endgroup
\end{corollary}
\begin{proof}
The proof follows Theorem \ref{T.3.3} by letting $\alpha=\beta=\gamma=2$, $\mathbb{D}_1=\mathbb{D}_2=[0, 1]$, $f_1(t)=f_2(t)=t$ and $\xi_1=\xi_2=0$.
\end{proof}
\begin{remark}\label{T.3.3.R.1}
We have
\begin{align*}
&\frac{1}{8}\left\|\,|a^*|^4 + |a|^4\right\|\,{\|a\|}^2
+ \frac{1}{8}\left\|\,|a^*|^2 + |a|^2\right\|\,\|a^2\|{\|a\|}^2 + \frac{1}{16}{\|a^2\|}^2{\|a\|}^2
\\& \qquad + \frac{1}{8}\left\|\,|a^*|^2 + |a|^2\right\|\,\|a\| \,v(a^3) + \frac{1}{8}\|a^2\|\,\|a\|\,v(a^3) + \frac{1}{16}v^2(a^3)
\\& \leq \frac{1}{4}{\|a\|}^6 + \frac{1}{4}{\|a\|}^6 + \frac{1}{16}{\|a\|}^6 + \frac{1}{4}{\|a\|}^6+\frac{1}{8}{\|a\|}^6+\frac{1}{16}{\|a\|}^6
= \|a\|^6.
\end{align*}
So, the inequality in Corollary \ref{T.3.3.C.1} is an improvement of the second inequality in \eqref{numerical radius inequality}.
\end{remark}
\begin{remark}\label{T.3.3.R.2}
Other improvements of the second inequality in \eqref{numerical radius inequality} can also be obtained by choosing other suitable
$\alpha,\beta,\gamma$ and $f_1, f_2$ in Theorem \ref{T.3.3}.
\end{remark}
\textbf{Conflict of interest.} The author declares that he has no conflict of interest.

\textbf{Data availability.} Data sharing not applicable to the present paper as no
data sets were generated or analyzed during the current study.
\bibliographystyle{amsplain}

\end{document}